\documentclass[12pt,reqno]{amsart}
\usepackage{amsmath, amssymb, latexsym, amscd, amsthm,amsfonts,amstext}
\usepackage{amsmath, amssymb, latexsym, amscd, amsthm,amsfonts,amstext}
\usepackage[mathscr]{eucal}
\usepackage{graphics}
\theoremstyle{plain}
\newtheorem{theorem}{Theorem}[section]

\newtheorem{proposition}[theorem]{Proposition}
\newtheorem{claim}[theorem]{\indent\sc Claim}
\theoremstyle{definition} 

\newcommand{\C}{{\mathbf{C}}}

\newcommand{\rank}{{\mathrm{rank}\,}}

\renewcommand{\P}{{\mathbf{P}}}

\renewcommand{\log}{\mathrm{log}}

\newcommand{\Z}{\mathbf{Z}}
\numberwithin{equation}{section}

\title[Two meromorphic mappings having the same inverse images]{Two meromorphic mappings having the same inverse images of moving hyperplanes} 

\begin{document}

\date{} 
\author{Si Duc Quang}
\address{Department of Mathematics\\
Hanoi National University of Education\\
136 XuanThuy str., Hanoi, Vietnam}
\email{quangsd@hnue.edu.vn}

\date{} 
\author{Le Ngoc Quynh}
\address{Faculty of Education\\
An Giang University\\
18 Ung Van Khiem, Dong Xuyen, Long Xuyen, An Giang, Vietnam}
\email{nquynh1511@gmail.com}

\thanks{This research was supported in part by a NAFOSTED grant of Vietnam (Grant No. 101.04-2015.03)}


\maketitle

\begin{abstract} 
In this paper, we will show that if two meromorphic mappings $f$ and $g$ of $\C^m$ into $\P^n(\C)$ have the same inverse images for $(2n+2)$ moving hyperplanes $\{a_i\}_{i=1}^{2n+2}$ with multiplicities counted to level $l_0$ then the map $f\times g$ must be algebraically degenerated over the field $\mathcal R\{a_i\}_{i=1}^{2n+2}$, where $l_0=3n^3(n+1)q(q-2)$ with $q=\binom{2n+2}{n+2}$. Our result generalizes the previous result for fixed hyperplanes case of Fujimoto and also improves his result by giving an explicit estimate for the number $l_0$. 
\end{abstract}

\section{Introduction}

In 1975, Fujimoto \cite{Fu75} showed that if two linearly nondegenerate meromorphic mappings $f$ and $g$ of $\C^m$ into $\P^n(\C)$ have the same inverse images of $3n+2$ hyperplanes in general position with counting multiplicity then $f=g$. In 1999, Fujimoto \cite{Fu99} considered the case where these inverse images are counted with multiplicities truncated by a level $l_0$. He proved the following theorem, in which the number $q$ of hyperplanes is also reduced.

\vskip0.2cm
{\sc Theorem A} (Fujimoto \cite[Theorem 1.5]{Fu99}).\ {\it Let $H_1,\ldots,H_{2n+2}$ be hyperplanes of $\P^n(\C)$ in general position. Then there exists an integer $l_0$ such that, for any two meromorphic mappings $f$ and $g$ of $\C^m$ into $\P^n(\C)$, if $\min(\nu_{(f,H_i)},l_0)=\min(\nu_{(g,H_i)},l_0)\ (1\le i\le 2n+2)$ then the mapping $f\times g$ into $\P^n(\C)\times \P^n(\C)$ is algebraically degenerate. }

\vskip0.2cm
Here, $f\times g$ is a mapping from $\C^m$ into $\P^n(\C)\times \P^n(\C)$ defined by
$$ (f\times g)(z)=(f(z),g(z))\in \P^n(\C)\times \P^n(\C)$$
for all $z$ outside the union of the indeterminacy loci of $f$ and $g$, $\nu_{(f,H_i)}$ is the pullback of divisor $H_i$ by $f$.  We also say that a meromorphic mapping into a projective variety is algebraically degenerate if its image is included in a proper analytic subset of the projective variety, otherwise it is algebraically non-degenerate.

We would like to emphasize here that the method used in the proof of Fujimoto is mainly base on the notion of logarithmic derivatives weight of rational functions. Hence his proofs is quite complicate and the number $l_0$ can not be estimated explicit. Then the following two questions arise naturally: 
\begin{itemize}
\item[1.] Are there any similar results to the above results of Fujimoto in the case where fixed hyperplanes are replaced by moving hyperplanes?
\item[2.] Does there exists an explicit estimate for the number $l_0$?
\end{itemize}
In this paper, by introducing new method, we will extend the above result of Fujimoto to the case of moving targets. We also simplify his proof and improve his result by giving an explicit estimate for the number $l_0$.  To state our result, we first recall the following. 

Let $f$ be a meromorphic mappings of $\C^m$ into $\P^n(\C)$ and let $a$ be a meromorphic mappings of $\C^m$ into $\P^n(\C)^*$. We say that $a$ is \textit{slowly moving hyperplanes} or \textit{slowly moving target} of $\P^n(\C)$ with respect to $f$ if $||\ T(r,a) = o(T(r,f))$ as $r \to \infty$ (see Section 1 for the notations). Similarly, a meromorphic function $\varphi$ on $\C^m$ is said to be ``small'' with respect to $f$ if $||\ T(r,\varphi) = o(T(r,f))$ as $r \to \infty$. Suppose that $f$ and $a$ have reduced representations $(f_0:\cdots :f_n)$ and $(a_0:\cdots :a_n)$ respectively. We set $(f,a)=\sum_{i=0}^na_if_i$. Hence the divisor of zeros $\nu_{(f,a)}$ of the function $(f,a)$ does not depend on the choice of these representations.

Let $a_1,\ldots,a_q$ $(q \geq n+1)$ be $q$ moving hyperplanes of $\P^n(\C)$ with reduced representations $a_i = (a_{i0}: \dots : a_{in})\ (1\le i \le q).$
We say that $a_1,\dots,a_q$ are located in general position if $\det (a_{i_kl}) \not \equiv 0$ for any $1\le i_0<i_1<\cdots<i_n\le q.$
We denote by $\mathcal {M}$ the field of all meromorphic functions on $\C^m$ and $\mathcal {R}\{a_i\}_{i=1}^q$ the smallest subfield of $\mathcal {M}$ which contains $\C$ and all $ {a_{jk}}/{a_{jl}}\text { with } a_{jl}\not\equiv 0$.

Let $N$ be a positive integer and let $V$ be a projective subvariety of $\P^N(\C)$. Take a homogeneous coordinates $(\omega_0:\cdots :\omega_N)$ of $\P^N(\C)$. Let $F$ be a meromorphic mapping of $\C^m$ into $V$ with a representation $F=(F_0:\cdots:F_N)$. 

\vskip0.2cm
{\sc Definition B.}{\it The meromorphic mapping $F$ is said to be algebraically degenerate over a subfield $\mathcal R$ of $\mathcal M$ if there exists a homogeneous polynomial $Q\in \mathcal R[\omega_0,\ldots,\omega_N]$ with the form
$$ Q(z)(\omega_0,\ldots,\omega_N)=\sum_{I\in\mathcal I_{d}}a_I(z)\omega^I, $$
 where $d$ is an integer, $\mathcal I_d=\{(i_0,\ldots,i_N)\ ;\ 0\le i_j\le d,\sum_{j=0}^Ni_j=d\},$ $a_I\in\mathcal R$ and $\omega^I=\omega_0^{i_0}\cdots\omega_N^{i_N}$ for $I=(i_0,\ldots,i_N)$, such that

$\mathrm{(i)}$\ $Q(z)(F_0(z),\ldots,F_N(z))\equiv 0\ \text{ on } \C^m,$

$\mathrm{(ii)}$ there exists $z_0\in\C^m$ with $Q(z_0)(\omega_0,\ldots,\omega_N)\not\equiv 0 \text{ on }V$.}

Now let $f$ and $g$ be two meromorphic mappings of $\C^m$ into $\P^n(\C)$ with representations
$$ f=(f_0:\cdots f_n)\text{ and }g=(g_0:\cdots :g_n). $$
We consider $\P^n(\C)\times \P^n(\C)$ as a projective subvariety of $\P^{(n+1)^2-1}(\C)$ by Segre embedding. Then the map $f\times g$ into $\P^n(\C)\times \P^n(\C)$ is algebraically degenerate over a subfield $\mathcal R$ of $\mathcal M$ if there exists a nontrivial polynomial
$$ Q(z)(\omega_0,\ldots,\omega_n,\omega_0',\ldots,\omega_n')=\sum_{\underset{i_0+\cdots +i_n=d}{I=(i_0,\ldots,i_n)\in \Z^{n+1}_{+}}}\sum_{\underset{j_0+\cdots +j_n=d'}{J=(j_0,\ldots,j_n)\in \Z^{n+1}_{+}}}a_{IJ}(z)\omega^{I}\omega'^{J}, $$
where $d,d'$ are positive integers, $a_{IJ}\in\mathcal R$, such that
$$ Q(z)(f_0(z),\ldots,f_n(z),g_0(z),\ldots,g_n(z))\equiv 0. $$

Our main theorem is stated as follows.

\begin{theorem}\label{1.1}
Let $f$ and $g$ be two meromorphic mappings of $\C^m$ into $\P^n(\C)$. Let $a_1,\ldots,a_{2n+2}$ be $(2n+2)$ meromorphic mappings of $\C^m$ into $\P^n(\C)^*$ in general position, which are slow with respect to $f$ and $g$. Let $l_0$ be a positive positive or infinite. Assume that $\min (\nu_{(f,a_i)},l_0)=\min (\nu_{(g,a_i)},l_0)\ (1\le i\le 2n+2)$. If $l_0\ge 3n^3(n+1)q(q-2)$, where $q=\binom{2n+2}{n+2}$, then the map $f\times g$ into $\P^n(\C)\times \P^n(\C)$ is algebraically degenerate over $\mathcal R\{a_i\}_{i=1}^{2n+2}$.
\end{theorem}

We would also like to emphasize here that there are many results on finiteness and degeneracy problem of meromorphic mappings with fixed or moving hyperplanes given by Ru \cite{Ru01}, Tu \cite{Tu02}, Thai-Quang \cite{TQ05}, Dethloff-Tan \cite{DT}, Giang-Quynh-Quang \cite{GQQ}\cite{QQ} and others. However, in all their results, they need two aditional assumption as follows.

(i) $\dim\{(f,a_i)=(f,a_j)=0\}\le m-2$ for all $i\ne j$.

(ii) $f=g$ on the union $\bigcup_{i}(f,a_i)^{-1}(0)$.\\
These assumptions play essential roles in their proof and are very complicate to verify. In our result, these assumptions are omitted.

{\sc Acknowledgments.} The research of the authors was supported in part by a NAFOSTED grant of Vietnam (Grant No. 101.04-2015.03).

\section{Basic notions and auxiliary results from Nevanlinna theory}

\noindent
{\bf (a)}\ We set $||z|| = \big(|z_1|^2 + \dots + |z_m|^2\big)^{1/2}$ for
$z = (z_1,\dots,z_m) \in \C^m$ and define
\begin{align*}
B(r) := \{ z \in \C^m ; ||z|| < r\},\quad
S(r) := \{ z \in \C^m ; ||z|| = r\}\ (0<r<\infty).
\end{align*}

Define 
$$v_{m-1}(z) := \big(dd^c ||z||^2\big)^{m-1}\quad \quad \text{and}$$
$$\sigma_m(z):= d^c \text{log}||z||^2 \land \big(dd^c \text{log}||z||^2\big)^{m-1}
 \text{on} \quad \C^m \setminus \{0\}.$$

A divisor $\nu$ on $\C^m$ is given by a formal sum $\nu=\sum\nu_{\mu}X_{\mu}$, where $\{X_\mu\}$ is a locally finite family of analytic hypersurfaces in $\C^m$ and $\nu_{\mu}\in \Z$. We may assume that $X_\mu$ are irreducible and distinct to each other, and $\nu_\mu \not= 0$ for all $\mu$. Then the set $\mathrm{Supp} (\nu)=\bigcup_\mu X_\mu $ is called the support of $\nu$. Sometimes, we identify the divisor $\nu$ with the function $\nu (z)$ from $\C^m$ into $\Z$ defined as follows: 
$$ \nu (z)=\sum_{X_\mu\ni z}\nu_\mu. $$
For positive integers $k,M$ or $M= \infty$, we set $\nu^{[M]}(z)=\min\ \{M,\nu(z)\}$ and define the counting function of $\nu$ by
\begin{align*}
\nu_{>k}^{[M]}(z)&= 
\begin{cases}
\nu^{[M]}(z)& \text { if } \nu(z)> k,\\
0 & \text { if } \nu(z)\le k,
\end{cases}\\
n(t)&=
\begin{cases}
\int\limits_{|\nu|\,\cap B(t)}
\nu(z) v_{m-1} & \text  { if } m \geq 2,\\
\sum\limits_{|z|\leq t} \nu (z) & \text { if }  m=1. 
\end{cases}
\end{align*}
Similarly, we define $n^{[M]}(t)$ and $n_{>k}^{[M]}(t).$

Define
$$ N(r,\nu)=\int_1^r \dfrac {n(t)}{t^{2m-1}}dt \quad (1<r<\infty).$$
Similarly, we define  $N(r,\nu^{[M]}), \ N(r,\nu_{>k}^{[M]})$
and denote them by $N^{[M]}(r,\nu), \ N_{>k}^{[M]}(r,\nu)$, respectively. For brevity, we will omit the character $^{[M]}$ if $M=\infty$.

Let $\varphi : \C^m \rightarrow \C $ be a nonzero meromorphic function. Denote by $\nu^0$ the divisor of zeros of $\varphi$ and set
$$N_{\varphi}(r)=N(r,\nu^0_{\varphi}), \ N_{\varphi}^{[M]}(r)=N^{[M]}(r,\nu^0_{\varphi}), \ N_{\varphi,> k}^{[M]}(r)=N_{>k}^{[M]}(r,\nu^0_{\varphi}).$$

\noindent
{\bf (b)}\ Let $f : \C^m \rightarrow \P^n(\C)$ be a meromorphic mapping.
For arbitrarily fixed homogeneous coordinates
$(\omega_0 : \cdots : \omega_n)$ on $\P^n(\C)$, we take a reduced representation
$f = (f_0 : \cdots : f_n)$, which means that each $f_i$ is a  
holomorphic function on $\C^m$ and 
$f(z) = \big(f_0(z) : \dots : f_n(z)\big)$ outside the analytic set
$\{ f_0 = \cdots = f_n= 0\}$ of codimension at least $2$.
Set $\Vert f \Vert = \big(|f_0|^2 + \dots + |f_n|^2\big)^{1/2}$.
The characteristic function of $f$ is defined by 
\begin{align*}
T(r,f) = \int\limits_{S(r)} \text{log}\Vert f \Vert \sigma_n -
\int\limits_{S(1)}\text{log}\Vert f\Vert \sigma_n.
\end{align*}

Let $a$ be a meromorphic mapping of $\C^m$ into $\P^n(\C)^*$ with reduced representation
$a = (a_0 : \dots : a_n)$. Setting $(f,a):=a_0f_0+\cdots +a_nf_n$. If $(f,a)\not \equiv 0,$ then we define
$$m_{f,a}(r)=\int\limits_{S(r)} \text{log}\dfrac {||f||\cdot ||a||}{|(f,a)|}\sigma_m -
\int\limits_{S(1)}\text{log}\dfrac {||f||\cdot ||a||}{|(f,a)|}\sigma_m,$$
where $\Vert a \Vert = \big(|a_0|^2 + \dots + |a_n|^2\big)^{1/2}$.

The first main theorem for moving hyperplanes (see \cite{NO}) states
$$T(r,f)+T(r,a)=m_{f,a}(r)+N_{(f,a)}(r).$$

Let $\varphi$ be a nonzero meromorphic function on $\C^m$, which are occasionally regarded as a meromorphic map into $\P^1(\C)$. The proximity function of $\varphi$ is defined by
$$m(r,\varphi):=\int_{S(r)}\log \max\ (|\varphi|,1)\sigma_m.$$

\noindent
{\bf (c)} As usual, by the notation $``|| \ P"$,  we mean the assertion $P$ holds for all $r \in [0,\infty)$ excluding a Borel subset $E$ of the interval $[0,\infty)$ with $\int_E dr<\infty$.

The following play essential roles in Nevanlinna theory.
\begin{theorem}[{\cite[Theorem 2.1]{RuW} and \cite[Theorem 2]{TQ08}}]\label{2.1}
Let $f=(f_0:\cdots :f_n)$ be a reduced representation of a meromorphic mapping $f$ of $\C^m$ into $\P^n(\C)$. Assume that $f_{n+1}$ is a holomorphic function with $f_0+\cdots +f_n+f_{n+1}=0$. If $\sum_{i\in I}f_i\ne 0$ for all $I\subsetneq \{0,\ldots,n+1\}$, then
$$ ||\ T(r,f)\le\sum_{i=0}^{n+1}N_{f_i}^{[n]}(r)+o(T(r,f)). $$
\end{theorem}

\begin{theorem}[{\cite{NO} and \cite[Theorem 5.5]{Fu99}}]\label{2.2}
Let $f$ be a nonzero meromorphic function on $\C^m.$ Then 
$$\biggl|\biggl|\quad m\biggl(r,\dfrac{\mathcal{D}^\alpha (f)}{f}\biggl)=O(\log^+T(r,f))\ (\alpha\in \Z^m_+).$$
\end{theorem}
\begin{theorem}[{\cite[Theorem 5.2.29]{NO}}]\label{2.3}
Let $f$ be a nonzero meromorphic function on $\C^m$ with a reduced representation $f=(f_0:\cdots :f_n)$. Suppose that $f_k\not\equiv 0$. Then 
$$ T(r,\frac{f_j}{f_k})\le T(r,f)\le \sum_{j=0}^nT(r,\frac{f_j}{f_k})+O(1).$$
\end{theorem}

\section{Proof of Theorem \ref{1.1}}

In order to prove  Theorem \ref{1.1}, we need the following algebraic propositions.

Let $H_1,\ldots,H_{2n+1}$ be $(2n+1)$ hyperplanes of $\P^n(\C)$ in general position given by
$$ H_i:\ \ x_{i0}\omega_0+x_{i1}\omega_1+\cdots +x_{in}\omega_n=0\ (1\le i\le 2n+1). $$
We consider the rational map $\Phi :\P^n(\C)\times \P^n(\C)\rightarrow \P^{2n}(\C) $ as follows:

For $v=(v_0:v_1\cdots :v_n),\ w=(w_0:w_1:\cdots :w_n)\in \P^n(\C)$, we define the value $\Phi (v,w)= (u_1:\cdots :u_{2n+1})\in\P^{2n}(\C)$ by
$$ u_i=\frac{x_{i0}v_0+x_{i1}v_1+\cdots +x_{in}v_n}{x_{i0}w_0+x_{i1}w_1+\cdots +x_{in}w_n}. $$
\begin{proposition}[{see \cite[Proposition 5.9]{Fu99}}]\label{3.1}
The map $\Phi$ is a birational map of $\P^n(\C)\times \P^n(\C)$ to $\P^{2n}(\C)$.
\end{proposition}

Now let $b_1,\ldots,b_{2n+1}$ be $(2n+1)$ moving hyperplanes of $\P^n(\C)$ in general position with reduced representations
$$ b_i=(b_{i0}:b_{i1}:\cdots: b_{in})\ (1\le i\le 2n+1). $$
Let $f$ and $g$ be two meromorphic mappings of $\C^m$ into $\P^n(\C)$ with reduced representations
$$ f=(f_0:\cdots :f_n)\ \text{ and }\ g=(g_0:\cdots :g_n). $$
Define $h_i={(f,b_i)}/{(g,b_i)}\ (1\le i\le 2n+1)$ and $h_I=\prod_{i\in I}h_i$ for each subset $I$ of $\{1,\ldots,2n+1\}.$
Set $\mathcal I=\{I=(i_1,\ldots,i_n)\ ;\ 1\le i_1<\cdots <i_n\le 2n+1\}$. Let $\mathcal R\{b_i\}_{i=1}^{2n+1}$ be the smallest subfield of $\mathcal M$ which contains $\C$ and $\{b_{il}/b_{is};b_{is}\not\equiv 0, 1\le i\le 2n+1,0\le l,s\le n\}$. We have the following proposition.
\begin{proposition}\label{3.2}
If there exist functions $A_I\in\mathcal R\{b_i\}_{i=1}^{2n+1}\ (I\in\mathcal I)$, not all zero, such that
$$ \sum_{I\in\mathcal I}A_Ih_I\equiv 0, $$
then the map $f\times g$ into $\P^n(\C)\times\P^n(\C)$ is algebraically degenerate over $\mathcal R\{b_i\}_{i=1}^{2n+1}$.
\end{proposition}
\begin{proof}
By changing the homogeneous coordinates of $\P^n(\C)$, we may assume that $b_{i0}\not\equiv 0\ (1\le i\le 2n+1)$. Since $b_1,\ldots,b_{2n+1}$ are in general position, we have $\det (b_{i_jk})_{0\le j,k\le n}\not\equiv 0$ for $1\le i_0<\cdots i_{n}\le 2n+1$. Therefore, the set
$$ S=\bigl (\bigcap_{I\in\mathcal I}\{z\in\C^m\ ;\ A_I(z)=0\}\bigl )\cup\bigcup_{1\le i_0<\cdots i_{n}\le 2n+1}\{z\in\C^m\ ;\ \det (b_{i_jk}(z))_{0\le j,k\le n}= 0\} $$
is a proper analytic subset of $\C^m$. Take $z_0\not\in S$ and set $x_{ij}=b_{ij}(z_0)$.

For $v=(v_0:v_1\cdots :v_n),\ w=(w_0:w_1:\cdots :w_n)\in \P^n(\C)$, we define the map $\Phi (v,w)= (u_1:\cdots :u_{2n+1})\in\P^{2n}(\C)$ as above.  By Proposition \ref{3.1}, $\Phi$ is a birational function. This implies that the function 
$$\sum_{I\in\mathcal I}A_I(z_0)\prod_{i\in I}\frac{\sum_{j=0}^{n}b_{ij}(z_0)v_j}{\sum_{j=0}^{n}b_{ij}(z_0)w_j}$$
is a nonzero rational function. It follows that 
$$ Q(z)(v_0,\ldots,v_n,w_0,\ldots,w_n)=\dfrac{1}{\prod_{i=1}^{2n+1}b_{i0}}\sum_{I\in\mathcal I}A_I(z)\left (\prod_{i\in I}\sum_{j=0}^{n}b_{ij}(z)v_j\right )\times \left (\prod_{i\in I^c}\sum_{j=0}^{n}b_{ij}(z)w_j\right )$$
is a nonzero polynomial with coefficients in $\mathcal R\{b_i\}_{i=1}^{2n+1}$, where $I^c=\{1,\ldots.,2n+1\}\setminus I$. By the assumption of the proposition, it is clear that
$$ Q(z)(f_0(z),\ldots,f_n(z),g_0(z),\ldots,g_n(z))\equiv 0. $$
Hence $f\times g$ is algebraically degenerate over $\mathcal R\{b_i\}_{i=1}^{2n+1}$.
\end{proof}

\begin{proposition}\label{new3.1}
Let $f$ be a meromorphic mapping of $\C^m$ into $\P^n(\C)$ and let $b_1,\ldots,b_{n+1}$ be moving hyperplanes of $\P^n(\C)$ in general position with reduced representations
$$ f=(f_0:\cdots :f_n),\ b_i=(b_{i0}:\cdots :b_{in})\ (1\le i\le n+1). $$
Then for each regular point $z_0$ of the analytic subset $\bigcup_{i=1}^{n+1}\{z\ ;\ (f,b_i)(z)=0\}$ with $z_0\not\in I(f)$, we have
$$\min_{1\le i\le n+1}\nu^0_{(f,b_i)}(z_0)\le\nu^0_{\det \Phi} (z_0),$$
where $I(f)$ denotes the indeterminacy set of $f$ and $\Phi$ is the matrix $(b_{ij}; 1\le i\le n+1, 0\le j\le n).$
\end{proposition}
\begin{proof}
Since $z_0\not\in I(f)$, we may assume that $f_0(z_0)\ne 0.$ We consider the system of equations
$$b_{i0}f_0+\cdots +b_{in}f_n=(f,b_i)\ (1\le i\le n+1).$$
Solving these equations, we obtain
$$ f_0=\frac{\det\Phi'}{\det\Phi}, $$
where $\Phi'$ is the matrix obtained from $\Phi$ by replacing the first column of $\Phi$ by $((f,b_i))_{1\le i\le 2n+1}$.
Therefore, we have
$$ \nu^0_{\det \Phi}(z_0)=\nu^0_{\det\Phi' }(z_0)\ge\min_{1\le i\le n+1}\nu^0_{(f,b_i)}(z_0). $$
The proposition is proved.
\end{proof}

\vskip0.5cm
{\sc Proof of Theorem \ref{1.1}.}
\ Assume that $f,g,a_i$ have reduced representations
$$ f=(f_0:\cdots :f_n),\ g=(g_0:\cdots :g_n),\ a_i=(a_{i0}:\cdots :a_{in}).$$
Without loss of generality, we may assume that $f$ and $g$ are linearly non-degenerate over $\mathcal R\{a_i\}_{i=1}^{2n+2}$, otherwise the map $f\times g$ will be algebraically degenerate over $\mathcal R\{a_i\}_{i=1}^{2n+2}$. Then by the second main theorem for moving target, we have
\begin{align*}
\bigl |\bigl |\ \dfrac{2n+2}{n+2}T(r,f)&\le\sum_{i=1}^{2n+2}N_{(f,a_i)}^{[n]}(r)+o(T(r,f))\\
&\le n\cdot\sum_{i=1}^{2n+2}N_{(g,a_i)}^{[1]}(r)+o(T(r,f))\\
&\le n(2n+2)(T(r,g))+o(T(r,f)).
\end{align*}
Then we have $||\ T(r,f)=O(T(r,g))$. Similarly, we also have $||\ T(r,g)=O(T(r,f))$.  

By changing the homogeneous coordinates of $\P^n(\C)$ if necessary, we may assume that $a_{i0}\not\equiv 0$ for all $1\le i\le 2n+2$. We set $\tilde a_{ij}=a_{ij}/a_{i0}$,
$$ \tilde a_i=(\dfrac{a_{i0}}{a_{i0}},\dfrac{a_{i1}}{a_{i0}},\ldots ,\dfrac{a_{in}}{a_{i0}}),\ (f,\tilde a_i)=\sum_{j=0}^n\tilde a_{ij}f_j\ \text{ and }\  (g,\tilde a_i)=\sum_{j=0}^n\tilde a_{ij}g_j.$$

We suppose contrarily that the map $f\times g$ is algebraically non-degenerate over $\mathcal R\{a_i\}_{i=1}^{2n+2}.$
Define $h_i={(f,\tilde a_i)}/{(g,\tilde a_i)}\ (1\le i\le 2n+2).$ Then
${h_i}/{h_j} ={\bigl ((f,\tilde a_i)\cdot (g,\tilde a_j)\bigl )}/{\bigl ((g,\tilde a_i)\cdot (f,\tilde a_j)\bigl )}$ does not depend on  the choice of representations of $f$ and $g$ . 
Since $\sum_{k=0}^n \tilde a_{ik}f_k-h_i\cdot \sum_{k=0}^n \tilde a_{ik}g_k=0\ (1\le i \le 2n+2),$ it implies that 
\begin{align}\label{3.3}
\Phi :=\det \ (\tilde a_{i0},\ldots,\tilde a_{in},\tilde a_{i0}h_i,\ldots,\tilde a_{in}h_i; 1\le i \le 2n+2)\equiv 0.
\end{align}

For each subset $I\subset \{1,2,\ldots,2n+2\},$ put $h_I=\prod_{i\in I}h_i, \tilde h_I=\prod_{i\in I}\frac{h_i}{h_1}$. Denote by $\mathcal {I}$ the set 
$$ \mathcal I=\{(i_1,\ldots,i_{n+1})\ ;\ 1\le i_1<\cdots <i_{n+1}\le 2n+2\}. $$
For each $I=(i_1,\ldots,i_{n+1})\in \mathcal {I}$, define 
\begin{align*}
A_I=(-1)^{{(n+1)(n+2)}/{2}+i_1+\cdots+i_{n+1}}&\times \det (\tilde a_{i_rl};1\le r \le n+1,0\le l \le n)\\
&\times\det (\tilde a_{j_sl};1\le s \le n+1,0\le l \le n),
\end{align*}
where $J=(j_1,\ldots,j_{n+1})\in \mathcal {I}$ such that $I \cup J =\{1,2,\ldots,2n+2\}.$ We have
$$ \sum_{I\in\mathcal I}A_Ih_I=0. $$

We take a partition of $\mathcal I=\mathcal I_1\cup\cdots\cup\mathcal I_k$ with the following properties:
\begin{itemize}
\item $\mathcal I_t\cap\mathcal I_s=\emptyset,$, $1\le t<s\le k$.
\item $\sum_{I\in\mathcal I_t}A_Ih_I=0$, $1\le t\le k$.
\item $\sum_{I\in \mathcal J}A_Ih_I\not\equiv 0$ for any proper subset $J$ of $\mathcal I_t$, $1\le t\le k$.
\end{itemize}
For each $1\le t\le k$, we set $n_t=\sharp \mathcal I_t-2$ and assume that $\mathcal I_t=\{I_{0t},\ldots, I_{(n_t+1)t}\}$. We denote by $F_t$ the meromorphic mapping from $\C^m$ into $\P^{n_t}(\C)$ with the presentation $(A_{I_{0t}}h_{I_{0t}}:\cdots :A_{I_{n_tt}}h_{n_tt})$.

For each $1\le i\le 2n+2$, we define $S(i)$ the set of all indices $j\ne i$ such that there exist $\mathcal I_t$, two elements $I,I'\in\mathcal I_t$ satisfying
$$ \dfrac{h_I}{h_{I'}}=\dfrac{h_i}{h_j}. $$

\begin{claim}\label{3.3} For each $1\le i\le 2n+2$, $\sharp S(i)\ge n+1$,
\begin{align*}
 T(r,\dfrac{h_j}{h_i})&\le \frac{q(q-2)(n+1)}{l_0+1}T(r,g)+o(T(r,g)),\ \forall j\in S(i),\\
\text{and }T(r,\dfrac{h_j}{h_i})&\le 2\frac{q(q-2)(n+1)}{l_0+1}T(r,g)+o(T(r,g)),\ \forall j\not\in S(i),
\end{align*}
where $q=\binom{2n+2}{n+1}$.
\end{claim} 

Without loss of generality, we prove the claim for $i=n+2$. Indeed, suppose contrarily that $\sharp S(n+2)\le n$, we may assume that $1,\ldots, n+1\not\in S(n+2)$. Put $I_{01}=(1,\ldots ,n+1)$ and suppose that $T_{01}\in\mathcal I_1$. Since $f\times g$ is algebraically non-degenerate over $\mathcal R{\{a_i\}_{i=1}^{2n+2}}$,
$$ \sum_{s=0}^{n_1+1}A_{I_{s1}}(z)\dfrac{W^{I_{s1}}(z)}{V^{I_{s1}}(z)}\equiv 0,\ \forall z\in\C^m, $$
where $W^I(z)=\prod_{i\in I}\sum_{j=0}^na_{ij}(z)w_j$ and $V^I(z)=\prod_{i\in I}\sum_{j=0}^na_{ij}(z)v_j$. We take a point $z_0$ which is not zero neither pole of any $A_I$, not pole of any $a_{ij}$, not in the indeterminacy loci of all $a_i$ and such that the family of hyperplanes $\{a_i(z_0)\}_{i=1}^{2n+2}$ is in general position in $\P^n(\C)$. For the point $\omega =(\omega_0:\cdots :\omega_n)\in\bigcap_{i=n+3}^{2n+2}a_i(z_0) $, we have
$$ \sum_{\underset{s=0}{I_{s1}\subset\{1,\ldots ,n+2\}}}^{n_1+1}A_{I_{s1}}(z_0)\dfrac{W^{I_{s1}}(z_0)}{V^{I_{s1}}(z_0)}\equiv 0,\ \forall v.$$
This yields that there exists $1\le s'\le n_1+1$ such that $I_{s'1}\subset\{1,\ldots ,n+2\}$, $I_{s'1}\ne I_{01}$. Therefore $n+2\in I_{s'1}$ and 
$$ \dfrac{h_{I_{01}}}{h_{I_{s'1}}}=\dfrac{h_{n+2}}{h_j}$$
for some $j\in\{1,\ldots ,n+1\}$. This contradicts the supposition that $1,\ldots, n+1\not\in S(n+2)$. Hence, we must have $\sharp S(n+2)\ge n+1$. 

Now for $j\in S(n+2)$, we may assume that there exist $\mathcal I_t$  two elements $I,I'\in\mathcal I_1$ satisfying
$$ \dfrac{h_I}{h_{I'}}=\dfrac{h_j}{h_{n+2}}. $$ 
We suppose that $F_1$ has a reduced representation $F_1=(uA_{I_{01}}h_{I_{01}}:\cdots :uA_{I_{n_11}}h_{I_{n_11}})$, where $u$ is a meromorphic function.
Thus, by the second main theorem we easily have

\begin{align*}
T(r,\frac{h_j}{h_{n+2}})&\le T(r,F_1)\le \sum_{i=0}^{n_1+1}N^{[n_1]}_{uA_{I_{i1}}h_{I_{i1}}}(r)+o(T(r,g))\\
&\le\sum_{i=0}^{n_1+1}(N^{[n_1]}_{A_{I_{i1}}h_{I_{i1}}}(r)+N^{[t]}_{{1}/{A_{I_{i1}}h_{I_{i1}}}}(r))+o(T(r,g))\\
&\le\sum_{i=0}^{n_1+1}(N^{[n_1]}_{h_{I_{i1}}}(r)+N^{[n_1]}_{{1}/{I_{i1}}}(r))+o(T(r,g))\\
&\le n_1\sum_{i=0}^{n_1+1}(N^{[1]}_{h_{I_{i1}}}(r)+N^{[1]}_{{1}/{I_{i1}}}(r))+o(T(r,g))\\
&\le n_1\sum_{I\in\mathcal I}(N^{[1]}_{h_I}(r)+N^{[1]}_{1/h_I}(r))+o(T(r,g))\\
&\le n_1\sum_{I\in\mathcal I}\sum_{i\in I}(N^{[1]}_{h_i}(r)+N^{[1]}_{1/h_i}(r))+o(T(r,g))\\
&= \frac{n_1q}{2}\sum_{i=1}^{2n+2}(N^{[1]}_{h_{i}}(r)+N^{[1]}_{{1}/{h_{i}}}(r))+o(T(r,g))\\
&\le \frac{n_1q}{2}\sum_{i=1}^{2n+2}N^{[1]}_{(g,a_i),>l_0}(r)+o(T(r,g))\\
&\le\frac{n_1q}{2(l_0+1)}\sum_{i=1}^{2n+2}N_{(g,a_i),>l_0}(r)+o(T(r,g))\\
&\le\frac{n_1q(2n+2)}{2(l_0+1)}T(r,g)+o(T(r,g))\\
&\le \frac{q(q-2)(n+1)}{l_0+1}T(r,g)+o(T(r,g)).
\end{align*}
For $j\not\in S(n+2)$, since $\sharp S(j)\ge n+2$, there exists $l\in S(j)\cap S(n+2)$. Hence, by the above proof, we have
$$ T(r,\dfrac{h_j}{h_{n+2}})\le T(r,\dfrac{h_j}{h_l})+T(r,\dfrac{h_l}{h_{n+2}})\le 2\frac{q(q-2)(n+1)}{l_0+1}T(r,g)+o(T(r,g)).$$
The claim is proved.

\vskip0.2cm
We now continue the proof of Theorem \ref{1.1}. We see that there exist functions $b_{ij}\in\mathcal R\{a_i\}_{i=1}^{2n+2}\ (n+2\le i\le 2n+2,1\le j\le n+1)$ such that
$$ \tilde a_i=\sum_{j=1}^{n+1}b_{ij}\tilde a_j. $$
By the identity (\ref{3.3}), we have
$$ \det \ (\tilde a_{i0},\ldots,\tilde a_{in},\tilde a_{i0}h_i,\ldots,\tilde a_{in}h_i; 1\le i \le 2n+2)\equiv 0 .$$
It easily implies that
$$ \det\ (\tilde a_{i0}h_i-\sum_{j=1}^{n+1}b_{ij}\tilde a_{j0}h_j,\ldots,\tilde a_{in}h_i-\sum_{j=1}^{n+1}b_{ij}\tilde a_{jn}h_j; n+2\le i\le 2n+2,1\le j\le n+1)\equiv 0. $$
Therefore, the matrix
$$ \Psi =\bigl (\tilde a_{i0}h_i-\sum_{j=1}^{n+1}b_{ij}\tilde a_{j0}h_j,\ldots,\tilde a_{in}h_i-\sum_{j=1}^{n+1}b_{ij}\tilde a_{jn}h_j; n+2\le i\le 2n+2\bigl ) $$
has the rank at most $n$.

Suppose that $\rank\Psi <n$. Then, the determinant of the square submatrix
$$\bigl (\tilde a_{i1}h_i-\sum_{j=1}^{n+1}b_{ij}\tilde a_{j0}h_j,\ldots,\tilde a_{in}h_i-\sum_{j=1}^{n+1}b_{ij}\tilde a_{jn}h_j; n+2\le i\le 2n+1\bigl )$$
 vanishes identically. By Proposition \ref{3.2}, it follows that $f\times g$ is algebraically degenerate over $\mathcal R\{a_i\}_{i=1}^{2n+2}$. This contradicts the supposition. Hence $\rank\Psi =n$.

Without loss of generality, we may assume that the determinant of the square submatrix
$$\bigl (\tilde a_{i1}h_i-\sum_{j=1}^{n+1}b_{ij}\tilde a_{j0}h_j,\ldots,\tilde a_{in}h_i-\sum_{j=1}^{n+1}b_{ij}\tilde a_{jn}h_j; n+2\le i\le 2n+1\bigl )$$
 of $\Psi$ does not vanish identically. On the other hand, we have
$$(\tilde a_{i0}h_i-\sum_{j=1}^{n+1}b_{ij}\tilde a_{j0}h_j)g_0+\cdots +(\tilde a_{in}h_i-\sum_{j=1}^{n+1}b_{ij}\tilde a_{jn}h_j)g_n=0\ (n+2\le i\le 2n+1) .$$
Thus
\begin{align*}
(\tilde a_{i0}\frac{h_i}{h_1}-\sum_{j=1}^{n+1}b_{ij}\tilde a_{j0}\frac{h_j}{h_1})\frac{g_0}{g_n}&+\cdots +(\tilde a_{in}\frac{h_i}{h_1}-\sum_{j=1}^{n+1}b_{ij}\tilde a_{j(n-1)}\frac{h_j}{h_1})\frac{g_{n-1}}{g_n}\\&=-\tilde a_{in}\frac{h_i}{h_1}+\sum_{j=1}^{n+1}b_{ij}\tilde a_{jn}\frac{h_j}{h_1}\ (n+2\le i\le 2n+1).
\end{align*}
We regard the above identities as a system of $n$ equations in unknown variables ${g_0}/{g_n},\ldots ,$ ${g_{n-1}}/{g_n}$ and solve these to obtain that ${g_i}/{g_n} \ (0\le i\le n-1)$ has the form
$$ \frac{g_i}{g_n}= \frac{P_i}{Q_i},$$
where $P_i$ and $Q_i$ are homogeneous polynomials in ${h_j}/{h_1}\ (1\le j\le 2n+1)$ of degree $n$ with coefficients in $\mathcal R\{a_j\}_{j=1}^{2n+2}$. Then by Theorem \ref{2.2} and the above claim we have
\begin{align*}
T&(r,g)\le\sum_{i=0}^{n-1}T(r,\frac{g_i}{g_n})=\sum_{i=0}^{n-1}T(r,\frac{P_i}{Q_i})\le n^2\sum_{j=1}^{2n+1}T(r,\dfrac{h_j}{h_1})+o(T(r,g))\\
&\le n^2\biggl (\sum_{\underset{j\in S(1)}{2\le j\le 2n+1}}\frac{q(q-2)(n+1)}{l_0+1}T(r,g)+ 2\sum_{\underset{j\not\in S(1)}{2\le j\le 2n+1}}\frac{q(q-2)(n+1)}{l_0+1}T(r,g) \biggl )+o(T(r,g)).\\
&\le 3n^3\frac{q(q-2)(n+1)}{l_0+1}T(r,g)+o(T(r,g)),
\end{align*}
where the last inequality comes from the fact that there are at most $n$ indices $j\not\in S(1)$.
Letting $r\rightarrow +\infty$, we get 
$$1\le 3n^3\frac{q(q-2)(n+1)}{l_0+1},\ \text{ i.e., }l_0+1\le 3n^3(n+1)q(q-2).$$
This is a contradiction. Hence, $f\times g$ is algebraically degenerate over $\mathcal R\{a_i\}_{i=1}^{2n+2}.$
The theorem is proved.\hfill$\square$

\end{document}